\documentclass[11pt,leqno]{amsart}
\usepackage[utf8]{inputenc}

\usepackage{amssymb,amsfonts,amscd,amsbsy,amsmath,amsthm,amsrefs}
\usepackage{graphicx}
\usepackage{float}
\usepackage{color}   
\usepackage[
colorlinks=true,
linkcolor=blue,
citecolor=blue,
urlcolor=blue,
]{hyperref}

\newtheorem{theorem}{Theorem}
\newtheorem{lemma}[theorem]{Lemma}

\title{On Kuzmin-Landau Lemma. }
\author{J. Arias de Reyna}
\address{Univ.~de Sevilla \\
Facultad de Matem\'aticas \&\ Imus \\
c/Tarfia, sn
 \\
41012-Sevilla \\
Spain} 
\email{arias@us.es}

\date{\today}
\begin{document}
\maketitle

Obtaining bounds for sum of exponentials $\sum e^{2\pi i a_k}$ with $a_k$ real numbers is crucial in many theorems on Number Theory. One of the most useful result of this type was proven in 1927 by 
Kuzmin, \cite{K}, in Russian. Kuzmin wrote: ``Inequalities of this type were introduced for the first time by Vinogradov.  Further proofs are found  in the work of Landau and van der Corput.  The author's proof is based on entirely different principles and gives a better bound.''

\begin{lemma}[Kuzmin version]\label{L:KL}
 Let $a_1< a_2< \cdots < a_n$ be real numbers such that the differences $\delta_k=a_{k+1}-a_k$ are increasing and satisfies $\theta\le \delta_1\le \cdots\le \delta_{n-1}\le 1-\theta$ for some $0<\theta\le \frac12$. 
Then we have
\begin{equation}
\Bigl|\sum_{k=1}^n e^{2\pi i a_k}\Bigr|\le \frac{2}{\sin\pi\theta}\le\frac{1}{\theta}.
\end{equation}
\end{lemma}

Kuzmin's proof was entirely geometrical and it was displayed into four pages of the article.  Not much later, in 1928, Landau gave another proof of Kuzmin Lemma \cite{L1}.  Landau's proof fitted into a footnote of his paper.  How on earth could Landau arrive at such a concise proof?  The answer is simple: he carefully read Kuzmin's geometrical proof and translated it into an arithmetical form.   Edmund Landau was a faithful disciple of Weierstrass, the master of the  \emph{arithmetization of Analysis}, which he taught in his  lectures and was exercised  by his disciples. Landau just took the task of placing a geometrical proof into  arithmetical terms.

Comparing these two proofs is a good opportunity to reflect on the powerful tool that the 
arithmetization of Analysis allowed. We first give Kuzmin's geometrical 
proof and then Landau's arithmetical proof.


\begin{proof}[Kuzmin's geometrical proof of Lemma \ref{L:KL}]
For $n<3$ the Lemma is easily proved, so assume $n\ge3$. Consider the polygonal chain with vertex at the partial sums
\[A_0=0,\quad A_m=\sum_{k=1}^m e^{2\pi i a_k}.\]
Let  $C_m$ be the center of the circle passing through $A_{m-1}$, $A_m$, $A_{m+1}$.
We bound the total sum $A_n$ in the following way 
\[\sum_{k=1}^n e^{2\pi i a_k}=A_n-A_0=A_n-C_{n}+\sum_{m=2}^{n} (C_{m}-C_{m-1})+C_1-A_0.\]

Let $M_m$ be the center of the segment $A_{m-1}A_m$. The angle 
$\theta_m:=\widehat{M_mC_mM}_{m+1}$ is equal to  the angle of the segments $A_{m-1}A_m$ and $A_{m}A_{m+1}$. Hence, it is equal to the angle between $e^{2\pi i a_{m+1}}$ and $e^{2\pi i a_{m}}$. Therefore $\theta_m=2\pi(a_{m+1}-a_m)$. The triangle  $C_mA_mM_m$ is rectangle, so that  $C_mM_m=\frac12\cot\frac{\theta_m}{2}$.  We have, by definition of $C_m$, that $C_mM_m=C_mM_{m+1}$. It follows that the length of the segment joining
$C_m$ and $C_{m+1}$ is the difference between two segments situated on  $C_mM_{m+1}$. Since 
the increments are increasing,  the difference $\cot\frac{\theta_{m}}{2}-\cot\frac{\theta_{m+1}}{2}>0$.  
Therefore
\[|C_{m+1}-C_m|=\frac12\cot\frac{\theta_{m}}{2}-\frac12\cot\frac{\theta_{m+1}}{2}>0.\]
\begin{figure}[H]
\begin{center}
\includegraphics[width=\hsize]{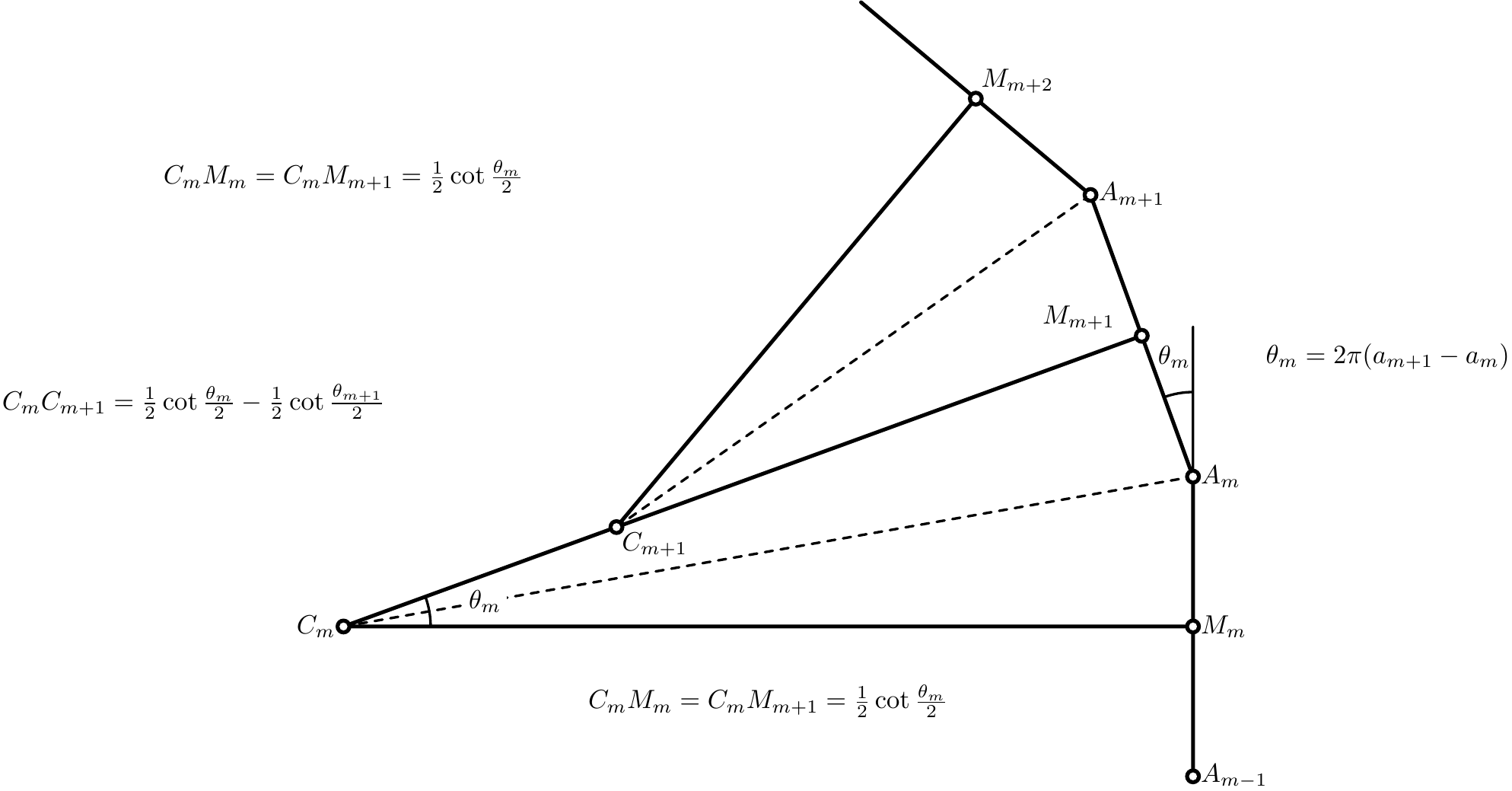}
\label{default}
\end{center}
\end{figure}
Notice also in the figure that $C_m$ is equidistant from   $A_{m-1}$, $A_m$ and $A_{m+1}$. Hence, the triangle $C_mA_mM_m$ shows that
\[|A_{m-1}-C_m|=|A_m-C_m|=|A_{m+1}-C_m|=\frac12\csc\frac{\theta_m}{2}.\]
Now, we may bound the total sum 
\begin{align*}
\Bigl|\sum_{k=1}^n e^{2\pi i a_k}\Bigr|&\le |A_n-C_{n-1}|+\sum_{m=2}^{n-1} |C_{m}-C_{m-1}|+|C_1-A_0|\\
&=\frac12\csc\frac{\theta_{n-1}}{2}+\sum_{m=2}^{n-1} (\frac12\cot\frac{\theta_{m-1}}{2}-\frac12\cot\frac{\theta_{m}}{2})+\frac12\csc\frac{\theta_1}{2}\\
&=\frac12(\csc\frac{\theta_{n-1}}{2}-\cot\frac{\theta_{n-1}}{2})+
\frac12(\cot\frac{\theta_{1}}{2}+\csc\frac{\theta_{1}}{2})\\
&=\frac{1-\cos\frac{\theta_{n-1}}{2}}{2\sin\frac{\theta_{n-1}}{2}}+\frac{1+\cos\frac{\theta_1}{2}}{2\sin\frac{\theta_1}{2}}.
\end{align*}
Since 
\[\pi\theta\le\frac{\theta_1}{2}\le\frac{\theta_{n}}{2}\le \pi(1-\theta),\]
it follows that 
\[\Bigl|\sum_{k=1}^n e^{2\pi i a_k}\Bigr|\le \frac{2}{\sin\pi\theta}\le \frac{1}{\theta}.\qedhere\]
\end{proof}


\begin{proof}[Landau's arithmetical  proof of Lemma \ref{L:KL}] 
The case $n=1$ is trivial. For $n\ge2$ put 
$b_k=\pi(a_k-a_{k-1})$ for $2\le k\le n$ and $e_k=e^{2\pi i a_k}$ for $1\le k\le n$, then
\begin{align*}
\sum_{k=1}^n e_k&=\Bigl(e_1+e_2\frac{ie^{-ib_2}}{2\sin b_2}\Bigr)+\sum_{k=2}^{n-1}
\Bigl(e_k-e_k\frac{ie^{-ib_k}}{2\sin b_k}+e_{k+1}\frac{ie^{-ib_{k+1}}}{2\sin b_{k+1}}\Bigr)+
\Bigl(e_n-e_n\frac{i e^{-ib_n}}{2\sin b_n}\Bigr)\\
&=e_1\Bigl(1+\frac{ie^{ib_2}}{2\sin b_2}\Bigr)+\sum_{k=2}^{n-1}e_k\Bigl(1-\frac{ie^{-ib_k}}{2\sin b_k}+\frac{ie^{ib_{k+1}}}{2\sin b_{k+1}}\Bigr)+e_n\Bigl(1-\frac{ie^{-ib_n}}{2\sin b_n}\Bigr)\\
&=e_1\frac{ie^{-i b_2}}{2\sin b_2}-\frac{i}{2}\sum_{k=2}^{n-1}e_k(\cot b_k-\cot b_{k+1})-
e_n\frac{ie^{ib_n}}{2\sin b_n},
\end{align*}
\begin{align*}
\Bigl|\sum_{k=1}^n e_k\Bigr|&\le \frac{1}{2\sin b_2}+\frac{1}{2}\sum_{k=2}^{n-1}(\cot b_k-\cot b_{k+1})+\frac{1}{2\sin b_n}\\
&=\frac{1+\cos b_2}{2\sin b_2}+\frac{1-\cos b_n}{2\sin b_n}<\frac{2}{\sin\pi\theta}\le \frac{1}{\theta}.\qedhere
\end{align*}
\end{proof}
For the benefit of the reader,  let us make three remarks.  To pass from the first to the second line, note that
\[
e_k e^{-i b_k}=e^{2\pi i a_k}e^{-\pi i (a_k-a_{k-1})}=e^{\pi i a_k+\pi i a_{k-1}}=
e^{2\pi i a_{k-1}}e^{\pi i (a_k-a_{k-1})}=e_{k-1}e^{ib_k}.\]
To pass from the second to the third line,  note that
\[1+\frac{i e^{ib}}{2\sin b}=1-\frac{e^{ib}}{2i\sin b}=\frac{e^{ib}-e^{-ib}-e^{ib}}{2i\sin b}=
\frac{ie^{ib}}{2\sin b}\]
\[\frac12+\frac{i e^{ib}}{2\sin b}=\frac12-\frac{e^{ib}}{2i\sin b}=\frac{\frac12e^{ib}-\frac12e^{-ib}-e^{ib}}{2i\sin b}=-\frac{\cos b}{2i\sin b}\]

An important point, when taking the absolute value in the forth line, note that 
\[|\cot b_k-\cot b_{k+1}|=\cot b_k-\cot b_{k+1},\]
because $b_{k+1}\ge b_k$ by hypothesis and $\cot x$ decreases in $(0,\pi)$ from $+\infty$ to $-\infty$.

This terse proof of Kuzmin's Lemma is not Landau's  most important contribution to the Lemma. In his note \cite{L2}, the  paper following his note  \cite{L1} in the same issue of the journal, he gives  the following sharp version of  Lemma \ref{L:KL}.

\begin{lemma}[Landau's version]
Let $a_1< a_2< \cdots < a_n$ with $n>1$ be real numbers such that the differences $\delta_k=a_{k+1}-a_k$ are increasing and satisfies $\theta\le \delta_1\le \cdots\le \delta_{n-1}\le 1-\theta$ for some $0<\theta\le \frac12$. 
\begin{itemize}
\item[(a)] We have
$\displaystyle{S:=\Bigl|\sum_{k=1}^n e^{2\pi i a_k}\Bigr|\le \cot\frac{\pi\theta}{2}}$.
\item[(b)] For $\theta=1/2$ and for each positive fraction $\theta<\frac12$ with odd numerator and denominator, 
there are choices of $a_1,\dots, a_n,$ for which  $S=\cot\frac{\pi\theta}{2}$.
\item[(c)] For any other values of $\theta$ with $0<\theta<\frac12$, we have $S<\cot\frac{\pi\theta}{2}$.
\item[(d)] For any $\theta$ with $0<\theta\le\frac12$ and each $\varepsilon>0$,
there are choices of $a_1,\dots, a_n,$ for which
$S>\cot\frac{\pi\theta}{2}-\varepsilon$.
\end{itemize}
\end{lemma}

\begin{proof}
(a) We only have to modify the last line of the  proof of Lemma 1. We have seen  that
\[\Bigl|\sum_{k=1}^n e_k\Bigr|\le \frac{1+\cos b_2}{2\sin b_2}+\frac{1-\cos b_n}{2\sin b_n}.\]
Since $\pi(1-\theta)\le b_k\le \pi\theta$, we have $|\cos b_k|\le \cos\pi\theta$  and $\sin b_k\ge\sin\pi\theta$, so that
\[S\le \frac{1+\cos \pi\theta}{\sin \pi\theta}=\cot\frac{\pi\theta}{2}.\]

(b) For $\theta=\frac12$, the choice $a_1=0$, $a_2=\frac12$, $a_3=1$, gives
\[|1+e^{\pi i}+e^{2\pi i}|=1.\]
For $0<\theta=\frac{2M+1}{2N+1}<\frac12$ with $M\ge0$ and $N\ge1$ integers, the choice
\[a_k=\begin{cases}
k\theta & \text{for $0\le k\le N-1$},\\
N\theta+(k-N)(1-\theta) & \text{for $N\le k\le 2N$},
\end{cases}\]
gives
\[\Bigl|\sum_{n=0}^{N-1}e^{2\pi i n\theta}+e^{2\pi i N\theta}\sum_{m=0}^{N}
e^{2\pi im(1-\theta)}\Bigr|= \Bigl|\sum_{n=0}^{N-1}e^{2\pi i n\theta}+\sum_{m=0}^{N}
e^{2\pi im\theta}\Bigr|\]
\[=\Bigl|\frac{1-e^{2\pi i N\theta}+1-e^{2\pi i (N+1)\theta}}{1-e^{2\pi i \theta}}\Bigr|=
\Bigl|\frac{1+e^{-\pi i\theta}+1+e^{\pi i \theta}}{1-e^{2\pi i \theta}}\Bigr|=
\frac{1+\cos\pi\theta}{\sin\pi\theta}.\]

(c) Assume that $0<\theta<\frac12$ and that 
\[S=\frac{1+\cos\pi\theta}{\sin\theta}.\]
Then the second  inequality in
\[\Bigl|\sum_{k=1}^n e_k\Bigr|\le \frac{1+\cos b_2}{2\sin b_2}+\frac{1-\cos b_n}{2\sin b_n}
\le \frac{1+\cos\theta}{2\sin\theta}+\frac{1+\cos\theta}{2\sin\theta}\]
is indeed an equality. Thus, $b_2=\pi\theta$ and $b_n=-\pi\theta$ (here  $n\ge3$). Hence, that not all differences
$b_3-b_2$, \dots, $b_n-b_{n-1}$, are null. Let $N>0$ be the first integer with $b_{N+2}-b_{N+1}>0$. Then $b_2=\cdots b_{N+1}=\pi\theta$. 

Since
\[
\Bigl|e_1\frac{ie^{-i b_2}}{2\sin b_2}-\frac{i}{2}\sum_{k=2}^{n-1}e_k(\cot b_k-\cot b_{k+1})-
e_n\frac{ie^{ib_n}}{2\sin b_n}\Bigr|=\frac{1}{2\sin b_2}+\frac{1}{2}\sum_{k=2}^{n-1}(\cot b_k-\cot b_{k+1})+\frac{1}{2\sin b_n}
\]
we have that
\[e_1\frac{ie^{-i b_2}}{2\sin b_2} \text{ and } -\frac{i}{2}e_{N+1}(\cot b_{N+1}-\cot b_{N+2})\]
lie in the same direction. Therefore $e_1 e^{-ib_2}=-e_{N+1}$, that is,
\[e^{2\pi i a_1}e^{-\pi i \theta}=-e^{2\pi i a_{N+1}}\]
\[\pi a_{N+1}-\pi a_1=\sum_{n=2}^{N+1}b_n=N\pi\theta.\]
Therefore 
\[e^{-\pi i \theta}=-e^{2\pi i N\theta},\]
That is, $e^{\pi i(2N+1)\theta}=-1$, and so $(2N+1)\theta$ is an odd integer.

(d) Let $0<\theta<\frac12$ and $\varepsilon>0$ be given. Pick some $\theta'$ a fraction with odd numerator and denominator, and  such that $\theta<\theta'<\frac12$, and $\cot\frac{\pi\theta'}{2}>\cot\frac{\pi\theta}{2}-\varepsilon$. The choice given in (b), for which $S=\cot\frac{\pi\theta'}{2}$,  satisfies $\theta\le a_2-a_1\le a_n-a_{n-1}\le 1-\theta$ and $S>\cot\frac{\pi\theta}{2}-\varepsilon$.
\end{proof}


Some comments are in order.  In the 1930s, Landau faced unpopularity amongst nazi's students in Göttingen. They rejected his un-German style, unbearable to their German feelings, evidenced in his definition of $\pi$ as $2\tau$, where $\tau>0$ is the first zero of $\cos x$.
Comparing Landau's version with  Kuzmin's original proof may help understanding Landau's unpopularity. The paradox is that few things were more German than the arithmetization of Analysis.

In 1958 Mordell gave a beautiful extension of these results, \cite{M}. Referring to results on this topic by  van del Corput, Kuzmin, Landau, Jarn\'{\i}k, Popken, Karamata and Tomic, Mordell wrote:
``All these authors prove their results geometrically except that Landau translates the geometrical argument into a transformation of series. Simple as his method is, it does not really
reveal what underlies these results.'' The paradox is that Mordell's method is entirely arithmetical.

Landau's version of the Kuzmin-Landau Lemma seems to be somewhat forgotten in 
today's  literature. 
References abound where the bound $\cot\frac{\pi\theta}{2}$ is substituted  by $\frac{1}{\theta}$. 
It is true that for many applications the extra information may not be needed. 
Proofs of  the best constant can be found; with no mention, however, to Landau's result (which is written in German).  
Since $\cot\frac{\pi\theta}{2}\le \frac{2}{\pi\theta}$, we have $|S|\le \frac{2}{\pi\theta}$.
Kuzmin and Landau  \cite{L2} proved that $A=2/\pi$ is the best possible constant
in $|S|\le \frac{A}{\theta}$. It has been claimed (in a paper published
in a reputed  journal)  that $|S|\le \frac{1}{\pi\theta}+1$; this is false.
Unfortunately, this false Lemma is often quoted. Part of my motivation for writing this note was
to recover attention on Landau's sharp result.



\begin{thebibliography}{999}



\bibitem{GK}
\textsc{Graham, S. W. \&\  Kolesnik, G}, \emph{van der Corput's method of exponential sums}. London Mathematical Society Lecture Note Series, 126. Cambridge University Press, Cambridge, 1991.




\bibitem{K}
\textsc{Kuzmin, R.}, \href{http://www.mathnet.ru/php/archive.phtml?jrnid=lfmo&wshow=issue&year=1927&volume=1&volume_alt=&issue=2&issue_alt=&option_lang=eng}{\emph{On certain trigonemetric inequalities}}, J. Soc. Phys.-Math. Leningrade, \textbf{1}, (1927), 233--239


\bibitem{L1}
\textsc{Landau, E.}, \href{http://www.digizeitschriften.de/dms/img/?PID=GDZPPN00250751X}{\emph{Über das Vorzeichen der Gaussschen Summe}}, Nachr. Ges. Wiss. Göttingen, (1928) 19--20.


\bibitem{L2}
\textsc{Landau, E.}, \href{http://www.digizeitschriften.de/dms/img/?PID=GDZPPN002507528}{\emph{Über eine trigonometrische Summe}}, Nachr. Ges. Wiss. Göttingen, (1928) 21--24.

\bibitem{M}
\textsc{Mordell, L.}, \href{http://matwbn.icm.edu.pl/ksiazki/aa/aa4/aa412.pdf}
{\emph{On the Kusmin-Landau inequality for exponential sums}}, Acta Arithmetica, 
\textbf{4} (1958) 3--9. 


\end{thebibliography}
\end{document}